\documentclass[a4paper, 12pt]{amsart}
\usepackage{amsfonts}

\usepackage{amsmath}

\def\bc{{\mathbb C}}

\def\bn{{\mathbb N}}

\def\br{{\mathbb R}}

\def\a{\alpha}
\def\b{\beta}

\def\l{\lambda} 

\def\m{\mu}
\def\p{\psi}
\def\n{\nu}

\def\t{\tau}
\def\f{\varphi} 
\def\v{\phi}
  
\def\w{\omega} 

\def\xb{{\mathbf{x}}}

\def\id{{\bf 1}\!\!{\rm I}}

\newtheorem{theorem}{Theorem}[section]

\newtheorem{proposition}[theorem]{Proposition}
\newtheorem{corollary}[theorem]{Corollary}
\theoremstyle{definition}
\newtheorem{definition}[theorem]{Definition}
\newtheorem{remark}[theorem]{Remark}
\newtheorem{example}[theorem]{Example}

\begin{document}
\date{2014-3-12}
\title{Relative ergodic properties of C*-dynamical systems}
\author{Rocco Duvenhage and Farrukh Mukhamedov}
\address{Department of Physics\\
University of Pretoria\\
Pretoria 0002\\
South Africa}
\email{rocco.duvenhage@up.ac.za}
\address{Department of Computational and Theoretical Sciences, Faculty of
Science\\
International Islamic University Malaysia, P.O. Box 141, 25710 Kuantan,
Pahang\\
Malaysia} \email{far75m@gmail.com,\ farrukh\_m@iium.edu.my}

\begin{abstract}
We study various ergodic properties of C*-dynamical systems inspired
by unique ergodicity. In particular we work in a framework allowing
for ergodic properties defined relative to various subspaces, and in
terms of weighted means. Our main results are characterizations of such
relative ergodic properties, ergodic theorems resulting from these properties,
and examples exhibiting these properties.

\end{abstract}

\maketitle

\section{Introduction}

In recent years there has been much activity to study various
ergodic theorems and ergodic properties of C*- and W*-dynamical
systems. See for example \cite{NSZ},\cite{Fid}, \cite{Fid2},
\cite{BDS} and \cite{AET} (also see \cite{Kr} for a general account of
ergodic theorems). When studying ergodic properties of such
a system, it has become clear that it is often necessary to work
relative to some subalgebra (or even some more general subspace) of
the C*- or W*-algebra involved. This is already a standard idea in
classical ergodic theory (see for example \cite{HK05}), and it has
indeed been important in recent work in the noncommutative theory as
well, as can be seen for example in \cite{AD}, \cite{FM},
\cite{Fid2}, \cite{AET} and \cite{D3}. In this paper we continue
this development, focussing in particular on ideas related to Abadie and
Dykema's work \cite{AD} on relative unique ergodicity, but as opposed to
\cite{AD}, not just relative to fixed point algebras.

Furthermore, the study and application of ergodic theorems have revealed
that it is often necessary that the ordinary Cesaro means be replaced
by weighted averages
\begin{equation}\label{a_k}
\sum_{k=0}^{n-1}a_kf(T^kx).
\end{equation}
Then it is for example natural to ask: is there a weaker summation than
Cesaro, ensuring unique ergodicity. In \cite{K} it has been
established that unique ergodicity implies uniform convergence of
\eqref{a_k}, when $\{a_k\}$ is a Riesz weight (see also \cite{I} for
similar results). In \cite{BLRT} similar problems were considered
for transformations of Hilbert spaces.

We note that weighted averages have been studied and applied at least
since the 1960's for $\mathbb{Z}$-actions in classical ergodic
theory; see for example \cite{Baxter} and \cite{Krengel}.
Furthermore, weighted averages remain relevant in current research. For
example, recently Dykema and Schultz used a weighted mean ergodic
theorem in an interesting way in operator theory \cite{DykS}, while
in \cite{K} an application to uniform distributions was considered.
This motivates the general study of weighted ergodic theorems. Therefore
much of the work in this paper, though not all, is done in terms of weighted
averages. We also refer the reader to \cite{JLO, cls} for other kinds of
weighted ergodic theorems in the classical and the noncommutative setting.

More generally, keep in mind \cite{BR} that the theory of quantum dynamical
systems provides a convenient mathematical description of irreversible
dynamics of an open quantum system (see \cite{AH}) and the investigation of
ergodic properties of such dynamical systems has had a considerable
growth. In the quantum setting, the matter is more complicated than in
the classical case. Some differences between classical and quantum
situations are pointed out in \cite{AH},\cite{NSZ}. This motivates
an interest to study dynamics of quantum systems (see
\cite{FR1,FR2,FV}). Therefore, it is then natural to address the
study of the possible generalizations to the quantum case of various
ergodic properties known for classical dynamical systems. In particular in
\cite{Av,LP,MT} a non-commutative notion of unique ergodicity was
defined, and certain properties were studied. Subsequently in \cite{AD}
a general notion of unique ergodicity for automorphisms of a
$C^*$-algebra relative to its fixed point subalgebra was
introduced. In \cite{AM} a generalization of such a notion for
positive mappings of $C^*$-algebras, and its characterization in
term of Riesz means are given.

Motivated by the discussion above, in this paper we study various relative
ergodic properties of C*-dynamical systems. The properties are to a large
extent inspired by the notion of unique ergodicity relative to the fixed point
space as introduced in \cite{AD}, but of a more general form, for example
allowing one to work relative to other spaces than just the fixed point space,
and in terms of weighted means and more general semigroup actions.

On the one hand we develop general theory, using for example techniques and
ideas from recent work flowing from Abadie and Dykema's paper \cite{AD} and
joinings. On the other hand we consider specific examples, in particular where
we have an action of $\mathbb{Z}$ on a reduced group C*-algebra, to illustrate
these ergodic properties.

General notions regarding weighted means are discussed in Section 2.
In the process general definitions regarding C*-dynamical systems
are also introduced. Much of this is used later on in the paper as
well. Weighted ergodic theorems are briefly treated in this section,
including a version of the recent result of \cite{AD,AM} related to
unique ergodicity relative to the fixed point space of a C*-dynamical system,
and a weighted version of an ergodic theorem from \cite{DS} for disjoint systems.

In line with the discussion of ergodic properties relative to subspaces
above, $\left(E,S\right) $-mixing is introduced in Section 3. This is a
generalization of the mixing condition introduced and studied in \cite{Fid2}.
This generalization is natural when the subspace relative to which one works
is not necessarily the fixed point space. Here $E$ is in fact a general
linear map from the algebra to itself, but typically in examples $E$
will project onto some subalgebra or operator system, which is then
the subspace relative to which we work as discussed above.
On the other hand, $S$ is a set of bounded linear functionals on the algebra
whose role will be explained later on. Section 3 focuses on a class of
examples for reduced group C*-algebras illustrating $\left(E,S\right)
$-mixing where, unlike the examples in \cite{Fid2}, $E$ need not be the
conditional expectation onto the fixed point algebra. We also briefly discuss
why it is relevant to consider different subalgebras for the same system.

These ideas set the stage for Section 4 where certain weaker ergodic
properties are discussed, namely unique
$\left(E,S\right)$-ergodicity and unique $\left(E,S\right)$-weak
mixing. Much of Section 4 is devoted to developing general theory
for these properties in the context of weighted means. In particular it is shown
how unique $\left(E,S\right)$-weak mixing of a system can be
characterized in terms of unique $\left(E,S\right) $-ergodicity of
the product of the system with itself, in analogy to the well-known
result in classical ergodic theory (see \cite{ALW}). Moreover, our
results extends well-known results of \cite{Wa,L}.

We conclude the paper with Section 5, devoted to higher order mixing
properties and higher order recurrence properties. These ideas have
received attention over the past few years, for example in
\cite{NSZ}, \cite{D0}, \cite{Fid}, \cite{BDS} and \cite{AET}, which
focused on developing general theory for asymptotically abelian
systems. In Section 5 however, we focus on a specific class of
examples, building on Section 3, where we have higher order mixing
properties, but as opposed to the previous literature, we don't have
asymptotic abelianness. These properties involve so-called multitime
correlations functions, which also appear in the physics literature
\cite{Pe}.

\section{Weighted means}

We follow a simple approach to weighted means which is well suited for our
purposes. For a more abstract approach the reader is referred to
\cite[Section 4]{LT}. To illustrate the ideas around weighted means that we
will study, we first consider the mean ergodic theorem in Hilbert space.
Then we turn to relative unique ergodicity in the sense of \cite{AD}, and
lastly we consider the case of disjoint systems (see \cite{DS}).

We first set up a simple abstract setting to study weighted means. It will
be convenient to work in terms of the following definition, which can be
viewed as a generalization of the concept of a F\o lner sequence. Note that
the integrals we use here are Bochner integrals (see for example
\cite[Appendix E]{Co}, and also \cite[Chapter II]{DU} and \cite[Chapter III]
{DunS} for background). Note that the semigroup $G$ below need not have an
identity.

\begin{definition}
Let $G$ be a topological semigroup with a right invariant measure $\rho $ on
its Borel $\sigma $-algebra, and let $X$ be a Banach space. Consider a net
$\left( f_{\iota }\right) \equiv \left( f_{\iota }\right) _{\iota \in I}$
indexed by some directed set $I$, where $f_{\iota }\in L^{1}\left( \rho
\right) $, $f_{\iota }:G\rightarrow \mathbb{R}^{+}=[0,\infty )$ and $\int
f_{\iota }d\rho \neq 0$. Assume furthermore that $f_{\iota }F$ is Bochner
integrable for all bounded Borel measurable $F:G\rightarrow X$, and that for
such $F$
\begin{equation*}
\lim_{\iota }\frac{\int f_{\iota }(g)\left[ F(g)-F(gh)\right] dg}{\int
f_{\iota }d\rho }=0
\end{equation*}
in the norm topology for all $h\in G$ (where $dg$ refers to integration with respect to the
measure $\rho $). Then we call $\left( f_{\iota }\right) $ a \emph{(right)
averaging net} for $\left( G,X\right) $. If we rather require the condition
\begin{equation*}
\lim_{\iota }\frac{\int f_{\iota }(g)\left[ F(g)-F(hg)\right] dg}{\int
f_{\iota }d\rho }=0
\end{equation*}%
for all $h\in G$, then we call $\left( f_{\iota }\right) $ a \emph{left
averaging net} for $\left( G,X\right) $
\end{definition}

The $f_{\iota }$ will act as the weights in the ergodic theorems below. All
of the integrals above are of course over the whole of $G$. Strictly
speaking we should say that $\left( f_{\iota }\right) $ is an averaging net
for $\left( G,\rho ,X\right) $, but for convenience we suppress the $\rho $
in this notation; no ambiguities will arise. In fact, we often even write
$\int f_{\iota }d\rho $ simply as $\int f_{\iota }$. When we do not specify whether
a given averaging net is right or left, it is assumed to be right. Keep in
mind that the requirement that $f_{\iota }F$ be Bochner integrable means in
particular that its range has to be separable, i.e. $f_{\iota }F$ has to be
strongly measurable (see \cite[Appendix E]{Co}). However, the Borel
measurable functions $f_{\iota }F$ are automatically strongly measurable if
$X$ is separable, or if $f_{\iota }F$ is continuous and either $G$ is
separable or $f_{\iota }$ has compact support. In most results in this paper
we use $X=\mathbb{C}$.

The following sufficient conditions for a net to be an averaging net, which
are independent of the Banach space, are easily shown, and in particular
shows how F\o lner nets form a special case of averaging nets:

\begin{proposition}
Let $G$ be a topological semigroup with a right invariant measure $\rho $ on
its Borel $\sigma $-algebra. Let $\left( \Lambda _{\iota }\right) $ be a
F\o lner net in $G$, i.e. $\Lambda _{\iota }$ is a compact set in the Borel
$\sigma $-algebra of $G$ with $0<\rho\left(\Lambda_{\iota}\right)<\infty$ and
\begin{equation*}
\lim_{\iota }\frac{\rho \left( \Lambda _{\iota }\bigtriangleup \left(
\Lambda _{\iota }g\right) \right) }{\rho \left( \Lambda _{\iota }\right) }=0
\end{equation*}
for all $g\in G$. Let $f_{\iota }:=\chi _{\Lambda _{\iota }}$ be the
characteristic (i.e. indicator) function of $\Lambda _{\iota }$ on $G$. Then
$\left( f_{\iota }\right) $ is an averaging net for $\left( G,X\right) $ for
any Banach space $X$.
\end{proposition}

\begin{proposition}
Let $G$ be a topological group with a right invariant measure $\rho $ on its
Borel $\sigma $-algebra, and let $X$ be a Banach space. Consider a net
$\left( \Lambda _{\iota },f_{\iota }\right) \equiv \left( \Lambda _{\iota
},f_{\iota }\right) _{\iota \in I}$ indexed by some directed set $I$, where
$\Lambda _{\iota }\subset G$ is Borel measurable, and $f_{\iota }\in
L^{1}\left( \Lambda _{\iota }\right) \,$\ (in terms of $\rho $ restricted to
$\Lambda _{\iota }$), such that $f_{\iota }:\Lambda _{\iota }\rightarrow
\mathbb{R}^{+}$ and $\int_{\Lambda _{\iota }}fd\rho \neq 0$. Assume
furthermore that
\begin{equation*}
\lim_{\iota }\frac{\int_{\Lambda _{\iota }\backslash \left( \Lambda _{\iota
}h\right) }f_{\iota }d\rho }{\int_{\Lambda _{\iota }}f_{\iota }d\rho }=0
\text{ and }\lim_{\iota }\frac{\int_{\Lambda _{\iota }\cap \left( \Lambda
_{\iota }h\right) }\left| f_{\iota }(g)-f(gh^{-1})\right| dg}{\int_{\Lambda
_{\iota }}f_{\iota }d\rho }=0
\end{equation*}%
for all $h\in G$. Define a function $f_{\iota }^{\prime }$ on $G$ by
$f_{\iota }^{\prime }(x)=f_{\iota }(x)$ for $x\in \Lambda _{\iota }$, and
$f_{\iota }^{\prime }(x)=0$ for $x\notin \Lambda _{\iota }$. Then
$\left(f_{\iota }^{\prime }\right) $ is an averaging net for $\left( G,X\right) $
for any Banach space $X$.
\end{proposition}

The examples below (for $G=\mathbb{R}$) can be checked by using the proposition above.

\begin{example}
Consider the case $G=\mathbb{R}$. Set $\Lambda_{n}:=[0,n]$ for $n=1,2,3,...$
, or even any real $n>0$. Let $f(t):=t^{s}$ for an $s>-1$. Setting
$f_{n}:=f|_{\Lambda_{n}}$, one can verify that $\left(\Lambda_{n},f_{n}\right)$
gives an averaging net for $\mathbb{R}$ as in Proposition 2.3.
\end{example}

\begin{example}
Similarly $\Lambda_{n}:=[1,n]$ for $n=2,3,...$, or even any real $n>1$,
along with $f(t)=t^{-1}$, gives an averaging net for $\mathbb{R}$.
\end{example}

\begin{example}
Lastly, $\Lambda_{n}:=[0,n]$ and $f_{n}(t):=\left( n-t\right) ^{s}$ for
$s>-1 $, give an averaging net for $\mathbb{R}$.
\end{example}

To illustrate how averaging nets work in a simple setting, we first consider
a weighted mean ergodic theorem in Hilbert space, and then apply it to the
examples above.

\begin{theorem}
Consider a topological semigroup $G$ with a right invariant measure $\rho $
on its Borel $\sigma $-algebra, a Hilbert space $H$, and an averaging net
$\left( f_{\iota }\right) $ for $\left( G,H\right) $. Let $U$ be a
representation of $G$ as contractions on $H$, such that $G\rightarrow
H:g\mapsto U_{g}x$ is Borel measurable for all $x\in H$. Let $P$ be the
projection of $H$ onto the fixed point space $V$ of $U$, namely
\begin{equation*}
V:=\left\{ x\in H:U_{g}x=x\text{ for all }g\in G\right\} \text{.}
\end{equation*}
Then%
\begin{equation*}
\lim_{\iota }\frac{1}{\int f_{\iota }d\rho }\int_{\Lambda _{\iota }}f_{\iota
}(g)U_{g}xdg=Px
\end{equation*}
for all $x\in H$.
\end{theorem}

\begin{proof}
We follow a standard proof of the mean ergodic theorem (see for example
\cite[Section 2]{dBDS}), but we only give it in outline. For $x\in H$, set
\begin{equation*}
A_{\iota }(x):=\frac{1}{\int f_{\iota }d\rho }\int f_{\iota }(g)U_{g}xdg
\end{equation*}
For $x=y-U_{h}y$, for some $y\in H$ and $h\in G$, one has $\lim_{\iota
}A_{\iota }(x)=0$ by the properties of an averaging net. It then also
follows that for any $x\in N:=\overline{\text{span}\left\{ y-U_{h}y:y\in
H,h\in G\right\} }$ we have $\lim_{\iota }A_{\iota }(x)=0$. On the other
hand, for $x\in N^{\perp }=V$, we have $\lim_{\iota }A_{\iota }(x)=x$.
Combining these two facts, the result follows.
\end{proof}

\begin{example}
For Example 2.4, we obtain
\begin{equation*}
\lim_{n\rightarrow\infty}\frac{s+1}{n^{s+1}}\int_{0}^{n}t^{s}U_{t}xdt=Px
\end{equation*}
With a substitution and some manipulation this gives
\begin{equation*}
\lim_{n\rightarrow\infty}\frac{1}{n}\int_{0}^{n}U_{t^{1/(s+1)}}xdt=Px
\end{equation*}
\end{example}

\begin{example}
For Example 2.5, we obtain
\begin{equation*}
\lim_{n\rightarrow\infty}\frac{1}{\ln n}\int_{1}^{n}t^{-1}U_{t}xdt=Px
\end{equation*}
With a substitution and some manipulation this gives
\begin{equation*}
\lim_{n\rightarrow\infty}\frac{1}{n}\int_{0}^{n}U_{e^{t}}xdt=Px
\end{equation*}
\end{example}

\begin{example}
For Example 2.6, we obtain
\begin{equation*}
\lim_{n\rightarrow\infty}\frac{s+1}{n^{s+1}}\int_{0}^{n}\left( n-t\right) ^{s}U_{t}xdt=Px
\end{equation*}
which is a type of Voronoi average as discussed in \cite{K}, but now for the
group $\mathbb{R}$. In fact, more generally, if $f:\mathbb{R}\rightarrow
\mathbb{R}^{+}$ is such that $\left( \Lambda _{n},f|_{\Lambda _{n}}\right) $
gives an averaging net for $\mathbb{R}$ as in Proposition 2.3, then the functions $f_{n}(t):=f(n-t)$
also give an averaging net for $\mathbb{R}$ via $\left( \Lambda _{n},f_{n}\right)$.
\end{example}

With the basic framework of weighted means now in place, we turn to relative
unique ergodicity.

\begin{definition}
A \emph{C*-dynamical system} $\left( A,\alpha \right) $ consists of
a unital C*-algebra $A$ and an action $\alpha $ of a semigroup $G$
on $A$ as unital completely positive maps $\alpha _{g}:A\rightarrow
A$, i.e. as Markov operators. The \emph{fixed point operator system}
of a C*-dynamical system $\left( A,\alpha \right) $ is defined as
\begin{equation*}
A^{\alpha }:=\left\{ a\in A:\alpha _{g}(a)=a\text{ for all }g\in G\right\}
\text{.}
\end{equation*}
\end{definition}

By an \emph{operator system} of $A$, we mean a norm closed self-adjoint
vector subspace of $A$ containing the unit of $A$. Whenever we consider a
C*-dynamical system $\left( A,\alpha \right) $, the notation $G$ for the
semigroup is implied. Note that since $\alpha _{g}$ is positive and $\alpha
_{g}\left( 1\right) =1$, we have $\left\| \alpha _{g}\right\| =1$.

\begin{definition}
A C*-dynamical system $\left( A,\alpha \right) $ is called \emph{amenable}
if the following conditions are met: $G$ is a topological semigroup with a
right invariant measure $\rho $ on its Borel $\sigma $-algebra, and
furthermore $\left( G,A\right) $ has an averaging net $\left( f_{\iota
}\right) $. The function $G\rightarrow A:g\mapsto \alpha _{g}(a)$ is Borel
measurable for every $a\in A$.
\end{definition}

A central notion in our work will be that of an invariant state:

\begin{definition}
Given a C*-dynamical system $\left( A,\alpha \right) $, a state $\mu $ on $A$
is called an \emph{invariant} state of $\left( A,\alpha \right) $, or
alternatively an $\alpha $\emph{-invariant} state, if $\mu \circ \alpha
_{g}=\mu $ for all $g\in G$.
\end{definition}

\begin{definition}
We call the C*-dynamical system $\left( A,\alpha \right) $ \emph{uniquely
ergodic relative to }$A^{\alpha }$ if every state on $A^{\alpha }$ has a
unique extension to an invariant state of $\left( A,\alpha \right) $.
\end{definition}

We now consider a weighted version of the result of Abadie and Dykema
\cite[Theorem 3.2]{AD}, where the notion of relative unique ergodicity was
first introduced. The proof requires only minor modifications of that of
\cite[Theorem 3.2]{AD}. For example, even though $A^{\alpha }$ is in general
only an operator system, rather than a C*-algebra, virtually nothing in the
proof related to this aspect changes; one should just work in terms of norm
one projections (equivalently, positive projections, since they project onto
operator systems, which contain the unit) instead of conditional expectations (also see \cite{AM} and
\cite{FM}). Furthermore, in Theorem 2.15, complete positivity of $\alpha_g$
is not needed, just positivity, though we do use complete positivity in the
rest of the paper since we consider tensor products of systems.

Note that if we say that a norm one projection $E:A\rightarrow A^{\alpha }$
is $\alpha $\emph{-invariant}, we mean that $E\circ \alpha _{g}=E$ for all
$g\in G$, and similarly for linear functionals. Existence of limits, closures
etc. are all in terms of the norm topology on $A$. Also, when a right
invariant measure $\rho $ is also a left invariant measure on the Borel
$\sigma $-algebra of $G$, then we can call $G$ \emph{unimodular} with respect
to $\rho $.

\begin{theorem}\label{UE-A}
Let $\left( A,\alpha \right) $ be an amenable C*-dynamical system,
with $G$ unimodular with respect to the measure $\rho $, and let
$\left( f_{\iota }\right) $ be both a right and left averaging net
for $\left( G,A\right) $. Then statements (i) to (vi) below are
equivalent.
\begin{enumerate}
\item[(i)] The system $\left( A,\alpha \right) $ is uniquely ergodic
relative to $A^{\alpha }$.

\item[(ii)] The limit
\begin{equation*}
\lim_{\iota }\frac{1}{\int f_{\iota }d\rho }\int f_{\iota }\left( g\right)
\alpha _{g}(a)dg
\end{equation*}
exists for every $a\in A$.

\item[(iii)] The subspace $A^{\alpha}+$ span$\left\{ a-\alpha_{g}(a):g\in G,a\in
A\right\} $ is dense in $A$.

\item[(iv)] The equality $A=A^{\alpha}+\overline{\text{span}\left\{ a-\alpha
_{g}(a):g\in G,a\in A\right\} }$ holds.

\item[(v)] Every bounded linear functional on $A^{\alpha}$ has a unique bounded
$\alpha$-invariant extension to $A$ with the same norm.

\item[(vi)] There is a positive projection $E$ of $A$ onto some operator
system $B$ of $A$ such that for every $a\in A$ and $\f\in S(A)$, where $S(A)$ denotes the
set of all states on $A$, one has
\begin{equation*}
\lim_{\iota }\frac{1}{\int f_{\iota }d\rho }\int f_{\iota }\left(
g\right)\f(\alpha _{g}(a))dg=\f(E(a))
\end{equation*}
(in which case necessarily $B=A^\alpha$ and $\alpha_g\circ E=E=E\circ\alpha_g$
for all $g\in G$).
\end{enumerate}
Furthermore, statements (i) to (vi) imply the following statements:

\begin{enumerate}
\item[(vii)] There exists a unique $\alpha $-invariant positive projection
$E$ from $A$ onto $A^{\alpha }$.

\item[(viii)] The positive projection $E$ in (vii) is given by
\begin{equation*}
Ea=\lim_{\iota }\frac{1}{\int f_{\iota }d\rho }\int f_{\iota }\left(
g\right) \alpha _{g}(a)dg
\end{equation*}%
for all $a\in A$.
\end{enumerate}
\end{theorem}

The above theorem extends the results of \cite{AD,AM} for general
weights. Moreover, it has certain corollaries related to
multiparameter dynamical systems.

 We conclude this section by considering a weighted ergodic theorem
for disjoint systems, generalizing \cite[Theorem 3.3]{DS}. As with
the theorem above, this again relates to relative ergodic
properties, but unlike the theorem above the space relative to which
we work need not be a fixed point space. This result is also related
to the mean ergodic theorem (see for example \cite[Remark
3.11]{DS}). First we require some more definitions.

\begin{definition}
Consider a C*-dynamical system $\left( A,\alpha \right) $. If $B$ is an
$\alpha $-invariant C*-subalgebra of $A$, in other words $\alpha _{g}\left(
B\right) =B$ for all $g\in G$, and $B$ contains the unit of $A$, then we can
define $\beta _{g}:=\alpha _{g}|_{B}$ to obtain a C*-dynamical system
$\left( B,\beta \right)$ called a \emph{factor} of $\left( A,\alpha \right)$.
\end{definition}

\begin{definition}
Let $\left( A,\alpha \right) $ be a C*-dynamical system with an
invariant state $\mu$. Then we say that $\mathbf{A}
=\left( A,\alpha ,\mu \right) $ is a \emph{state preserving C*-dynamical
system}.
\end{definition}

We now generalize definitions from \cite{DS} regarding joinings to the
current case (in \cite{DS} only $\ast $-automorphism were considered, not
unital completely positive maps in general). The definitions are in fact of
exactly the same form, but now stated for the definition above. Note that
$\otimes _{m}$ denotes the maximal C*-algebraic tensor product. Given two
C*-dynamical systems $\left( A,\alpha \right) $ and $\left( B,\beta \right) $
with actions of the same semigroup $G$, and with every $\alpha _{g}$ and
$\beta _{g}$ completely positive, we define $\alpha \otimes _{m}\beta $ by
$\left( \alpha \otimes _{m}\beta \right) _{g}:=\alpha _{g}\otimes _{m}\beta
_{g}$, which is also completely positive, for all $g\in G$, to obtain the
C*-dynamical system $\left( A\otimes _{m}B,\alpha \otimes _{m}\beta \right) $.

\begin{definition}
Let $\mathbf{A}=\left( A,\alpha,\mu\right) $ and $\mathbf{B}=\left(
B,\beta,\nu\right) $ be state preserving C*-dynamical systems. A
\emph{joining} of $\mathbf{A}$ and $\mathbf{B}$ is an invariant state $\omega$ of
$\left( A\otimes_{m}B,\alpha\otimes_{m}\beta\right) $ such that
$\omega\left(a\otimes1\right) =\mu(a)$ and $\omega\left( 1\otimes b\right) =\nu(b)$ for
all $a\in A$ and $b\in B$. The set of all joinings of $\mathbf{A}$ and
$\mathbf{B}$ is denoted by $J(\mathbf{A},\mathbf{B})$. Consider a factor
$\left(R,\rho\right)$ of $\left( A\otimes_{m}B,\alpha\otimes _{m}\beta\right)$,
and a $\rho$-invariant state $\psi$ on $R$ which has at least one extension
to a joining of $\mathbf{A}$ and $\mathbf{B}$. So we obtain a state
preserving C*-dynamical system $\mathbf{R}=\left( R,\rho ,\psi\right) $.
Denote by $J_{\mathbf{R}}\left( \mathbf{A},\mathbf{B}\right) $ the subset of
elements $\omega$ of $J\left( \mathbf{A},\mathbf{B}\right) $ such that $%
\omega|_{R}=\psi$. If $J_{\mathbf{R}}\left( \mathbf{A},\mathbf{B}\right) $
contains exactly one element, then we say that $\mathbf{A}$ and $\mathbf{B}$
are \emph{disjoint relative to} $\mathbf{R}$.
\end{definition}

Note that in this definition we are working relative to $R$, or more precisely
$\mathbf{R}$, in keeping with the main theme of the paper. In particular,
$\mathbf{A}$ and $\mathbf{B}$ being disjoint relative to $\mathbf{R}$,
is an ergodic property of the pair of systems $\mathbf{A}$ and $\mathbf{B}$
leading to the ergodic theorem below. See \cite{DS} for a more concrete
discussion involving a relative version of weak mixing versus compactness
as well as examples.

Below, the notation $\mathbf{A}$, $\mathbf{B}$ and $\mathbf{R}$ will refer
to triples $\left( A,\alpha ,\mu \right) $, $\left( B,\beta ,\nu \right) $
and $\left( R,\rho ,\psi \right) $ respectively. We also need the following:

\begin{definition}
Let $A$ and $B$ be unital C*-algebras with states $\mu $ and $\nu $
respectively. A \emph{coupling} of the pairs $\left( A,\mu \right) $ and
$\left( B,\nu \right) $ is a state $\kappa $ on $A\otimes _{m}B$ such that
$\kappa \left( a\otimes 1\right) =\mu (a)$ and $\kappa \left( 1\otimes
b\right) =\nu (b)$ for all $a\in A$ and $b\in B$. If furthermore $\psi $ is
a state on a C*-subalgebra $R$ of $A\otimes _{m}B$ such that $\kappa
|_{R}=\psi $, then we call $\kappa $ a \emph{coupling} of $\left( A,\mu
\right) $ and $\left( B,\nu \right)$ \emph{relative to}
$\left(R,\psi\right)$.
\end{definition}

\begin{theorem}
Let $G$ be a topological semigroup with a right invariant measure $\rho $ on
its Borel $\sigma $-algebra, and let $\left( f_{\iota }\right) _{\iota \in
I} $ be an averaging net for $\left( G,\mathbb{C}\right) $. Let $\mathbf{A}$
and $\mathbf{B}$ be state preserving C*-dynamical systems (for actions of $G$
) which are disjoint relative to $\mathbf{R}$. Let $\left( \kappa
_{\iota }\right) _{\iota \in I}$ be a net of couplings of $\left( A,\mu
\right) $ and $\left( B,\nu \right)$ relative to $\left( R,\psi \right)
$. Assume that for every $\iota \in I$ the function
$G\rightarrow \mathbb{C}:g\mapsto \kappa _{\iota }\circ \left( \alpha
_{g}\otimes _{m}\beta _{g}\right) (a\otimes b)$ is measurable for all $a\in
A $ and $b\in B$. Then
\begin{equation}
\lim_{\iota }\frac{1}{\int f_{\iota }d\rho }\int f_{\iota }(g)\kappa _{\iota
}\left( \alpha _{g}\otimes _{m}\beta _{g}(c)\right) dg=\omega (c)  \tag{2.1}
\end{equation}%
for all $c\in A\otimes _{m}B$, where $\omega $ is the unique element of
$J_{\mathbf{R}}\left( \mathbf{A},\mathbf{B}\right) $.
\end{theorem}

\begin{proof}
The proof follows the same plan as the proof of \cite[Theorem 3.3]{DS}.

>From Lebesgue's dominated convergence theorem it follows that $g\mapsto
f_\iota (g)\kappa_{n}\circ\left( \alpha_{g}\otimes_{m}\beta_{g}\right) (c)$ is
integrable on $G$ for all $c\in A\otimes_{m}B$, which means that
the integrals in (2.1) indeed exist.

We define a net of states $\left( \omega _{\iota }\right) _{\iota \in I}$ on
$A\otimes _{m}B$ by
\begin{equation*}
\omega _{\iota }(c):=\frac{1}{\int f_{\iota }d\rho }\int f_{\iota }(g)\kappa
_{\iota }\left( \alpha _{g}\otimes _{m}\beta _{g}(c)\right) dg
\end{equation*}%
which then has a cluster point $\omega ^{\prime }$ in the weak* topology in
the compact set of all states on $A\otimes _{m}B$. Since $\kappa _{\iota }$
is a coupling, so is $\omega _{\iota }$, from which one can easily show that
$\omega ^{\prime }$ is also a coupling of $\left( A,\mu \right) $ and $
\left( B,\nu \right) $.

For any $h\in G$ and $c\in A\otimes _{m}B$
\begin{equation}
\lim_{\iota }\left[ \omega _{\iota }(c)-\omega _{\iota }\circ \left( \alpha
_{h}\otimes _{m}\beta _{h}\right) (c)\right] =0  \tag{2.2}
\end{equation}
since $\left( f_{\iota }\right) $ is an averaging net. Now, for an arbitrary
$\varepsilon >0$, consider the following weak* neighbourhood of $\omega
^{\prime }$:
\begin{equation*}
N:=\left\{ \theta \in S:\left| \theta (c)-\omega ^{\prime }(c)\right|
<\varepsilon \text{ and }\left| \theta \left( \alpha _{h}\otimes _{m}\beta
_{h}(c)\right) -\omega ^{\prime }\left( \alpha _{h}\otimes _{m}\beta
_{h}(c)\right) \right| <\varepsilon \right\}
\end{equation*}
By (2.2) we know that there is an $\iota _{0}\in I$ such that
\begin{equation*}
\left| \omega _{\iota }(c)-\omega _{\iota }\circ \left( \alpha _{h}\otimes
_{m}\beta _{h}\right) (c)\right| <\varepsilon
\end{equation*}
for $\iota >\iota _{0}$, but since $\omega ^{\prime }$ is a cluster point of
$\left( \omega _{\iota }\right) $, there is an $\iota _{1}>\iota _{0}$ such
that $\omega _{\iota _{1}}\in N$. Combining these facts,
\begin{equation*}
\left| \omega ^{\prime }(c)-\omega ^{\prime }\circ \left( \alpha _{h}\otimes
_{m}\beta _{h}\right) (c)\right| <3\varepsilon
\end{equation*}
which means that
\begin{equation*}
\omega ^{\prime }\circ \left( \alpha _{h}\otimes _{m}\beta _{h}\right)
(c)=\omega ^{\prime }(c)
\end{equation*}
and so we have shown that $\omega ^{\prime }\in J\left( \mathbf{A},\mathbf{B}
\right) $.

If $c\in R$, then $\alpha_{h}\otimes_{m}\beta_{h}(c)\in R$, since
$\left(R,\rho\right) $ is a factor of $\left(A\otimes_{m}B,\alpha\otimes_{m}\beta\right) $,
so from the definition of $\omega_{\iota}$ it follows that $\omega_{\iota}(c)=\psi(c)$
and therefore $\omega(c)=\psi(c)$. This proves that
$\omega^{\prime}\in J_{\mathbf{R}}\left( \mathbf{A},\mathbf{B}\right) =\left\{ \omega\right\} $,
in other words the net $\left( \omega_{\iota }\right) $ has $\omega$ as its unique
cluster point, therefore w*-$\lim _{\iota}\omega_{\iota}=\omega$, so in
particular $\lim_{\iota}\omega_{\iota }(c)=\omega(c)$ for all $c\in
A\otimes_{m}B$, as required.
\end{proof}

One can apply this theorem to pairs of disjoint W*-dynamical systems to
obtain the corresponding weighted versions of \cite[Theorem 3.8 and
Corollary 3.9]{DS}.

\section{$(E,S)$-mixing}

In \cite{Fid2} the following type of mixing condition (inspired by
condition (vi) in Theorem 2.15) was studied for C*-dynamical systems
$\left(A,\alpha\right)$ with an action of
$\mathbb{N=}\left\{1,2,3,...\right\}$:
\begin{equation*}
\lim_{n\rightarrow \infty }\varphi \left( \alpha ^{n}(a)\right)
=\varphi \left( Ea\right)
\end{equation*}
for all states $\varphi $ on the C*-algebra, where $E:A\rightarrow
A$ is some linear map. The fact that $\alpha $ does not act on $Ea$,
implicitly means that we are thinking in terms of the fixed point
operator system of $\alpha $, in particular $E$ will typically be a
projection onto the fixed point operator system, i.e. we are
considering an ergodic property relative to the fixed point operator
system.

However, in classical ergodic theory it is also natural to consider
ergodic properties relative to spaces other than the fixed point
space. One example of such a study in the case of mixing in
classical ergodic theory can be found in \cite{Rud}, although this
was for a single invariant measure, rather than the form inspired by
relative unique ergodicity that we are interested in here. And in
the last part of Section 2 regarding disjoint systems we already saw
a case of this in the noncommutative theory.

We therefore look at the following more general ergodic property:

\begin{definition}
Let $\left(A,\alpha \right)$ be a C*-dynamical system for an action
of the semigroup $\mathbb{N}$. Let $E:A\rightarrow A$ be linear, and
let $S$ be a set of bounded linear functionals on $A$. Then
$\left(A,\alpha \right)$ is said to be $\left(E,S\right)
$\emph{-mixing} if
\begin{equation*}
\lim_{n\rightarrow \infty }\varphi \left( \alpha ^{n}(a-Ea)\right)
=0
\end{equation*}
for all $a\in A$ and all $\varphi \in S$.
\end{definition}

Note that the point here is that one can now look for examples of
this property where $Ea$ need not be fixed under $\alpha $, and one
can consider certain classes of states rather than all states.

We consider a class of examples similar to those studied in
\cite[Section 3] {D2}, except that here we work in terms of reduced
group C*-algebras rather than group von Neumann algebras. (It is
indeed closely related to examples considered in \cite{Fid2} and
\cite{FM}.)

Let $\Gamma $ be any group (to which we assign the discrete
topology). Let $\lambda $ be the left regular representation of
$\Gamma $ on the Hilbert space $H:=L^{2}(\Gamma )$ defined in terms
of the counting measure on $\Gamma $, i.e. $\left[ \lambda
(g)f\right] (h):=f\left( g^{-1}h\right) $ for all $g,h\in \Gamma $
and $f\in H$. Let $A:=C_{r}^{\ast }\left( \Gamma \right) $ be the
reduced group C*-algebra, i.e. the C*-subalgebra of $B(H)$ generated
by $\left\{ \lambda \left( g\right) :g\in \Gamma \right\} $. Given a
group automorphism $T:\Gamma \rightarrow \Gamma $, we can define a
$\ast $-automorphism $\alpha $ of $A$ such that $\alpha \left(
\lambda \left( g\right) \right) =\lambda \left( Tg\right) $ (see
\cite[Section 3]{D2} for more details), giving us a C*-dynamical
system for an action of the group $\mathbb{Z}$, which we will call
the \emph{dual system of} $\left( \Gamma,T\right) $.

Given this situation, we define
\begin{equation*}
F:=\left\{ g\in \Gamma :T^{\mathbb{N}}g\text{ is finite}\right\}
\end{equation*}
where $T^{\mathbb{N}}g:=\left\{ T^{n}g:n\in \mathbb{N}\right\} $
(one could use $\mathbb{Z}$ instead of $\mathbb{N}$; it makes no
difference). Then $F$ is a subgroup of $\Gamma $ consisting of all
the elements with finite orbits under $T$. Let $B$ be the
C*-subalgebra of $A$ generated by $\left\{ \lambda \left( g\right)
:g\in F\right\} $. We call $B$ the \emph{finite orbit algebra of}
$\left( A,\alpha \right) $. Note that orbits of all elements of $B$
are not necessarily finite under $\alpha $; the name ``finite
orbit'' is used simply because $B$ originates from elements of
$\Gamma $ with finite orbits.

We are going to consider states $\varphi$ on $A$ above given by
density matrices, i.e. states $\varphi$ given by $\varphi(a)=$
Tr$\left( \rho a\right) $ where Tr is the usual trace on $B(H)$, and
$\rho$ is a density matrix on $H$, i.e. a trace-class operator
$\rho\geq0$ in $B(H)$ such that Tr $\left( \rho\right) =1$.

The notation introduced in the last three paragraphs will remain
fixed for the rest of this section. Furthermore we define $\delta
_{g}\in H$ by
\begin{equation*}
\delta _{g}\left( h\right) :=\left\{
\begin{array}{c}
1\text{ when }h=g \\
0\text{ otherwise}
\end{array}
\right.
\end{equation*}
for all $g,h\in \Gamma $.

Note that it is easy to obtain various systems of this sort. For
example we could take $\Gamma$ to be the free group on some set of
symbols (and this set can be arbitrary), and $T$ can then be
obtained from any bijection of the set of symbols to itself. In
particular we can obtain examples where the finite orbit algebra is
strictly larger than the fixed point algebra, and yet not the whole
of $A$.

Such dual systems provide us with examples of $\left( E,S\right)
$-strong mixing (and hence also of weaker properties like unique
$\left(E,S\right)$-ergodicity and unique $\left(E,S\right)$-weak
mixing in the next section):

\begin{theorem}
Let $T:\Gamma \rightarrow \Gamma $ be any automorphism of an
arbitrary group $\Gamma $. Let $\left( A,\alpha \right) $ be the
dual system of $\left(\Gamma ,T\right) $. Then there exists a
conditional expectation $E:A\rightarrow B$ of $A$ onto the finite
orbit algebra $B$ of $\left(A,\alpha \right) $ such that
\begin{equation*}
E\lambda \left( g\right) =\left\{
\begin{array}{c}
\lambda \left( g\right) \text{ when }g\in F \\
0\text{ otherwise}
\end{array}
\right.
\end{equation*}
for all $g\in \Gamma $. Furthermore, $\left( A,\alpha \right) $ is
$\left(E,S\right) $-mixing, where $S$ is the set of all states on
$A$ given by density matrices on $H$.
\end{theorem}

\begin{proof}
We apply Tomita-Takesaki theory to the corresponding group von
Neumann algebra to prove this result.

(i) First we show the existence of $E$.

Let $M$ be the group von Neumann algebra of $\Gamma $, and let $N$
be the von Neumann algebra generated by $B$, in other words by
$\left\{ \lambda \left( g\right) :g\in F\right\} $. We also use the
notation $\Omega :=\delta _{1}$ where $1$ denotes the identity
element of $\Gamma $. Define a state $\omega $ on $M$ by
\begin{equation*}
\omega \left( a\right) :=\left\langle \Omega ,a\Omega \right\rangle
\end{equation*}
for all $a\in M$. It is straightforward to show that $\Omega $ is
cyclic and separating for $M$, and that $\omega $ is a trace. In
particular this means that the modular group of $\omega $ is
trivial, i.e. $\sigma _{t}^{\omega }=$ id$_{M}$ for all $t\in
\mathbb{R}$, and therefore $\sigma _{t}^{\omega }\left( N\right) =N$
which means that there is a unique conditional expectation
$D:M\rightarrow N$ such that $\omega \circ D=\omega $.

This conditional expectation is given by $\left( Da\right) \Omega
=Pa\Omega $ where $P$ is the projection of $H$ onto
$\overline{N\Omega }$ (see for example \cite[Section 10.2]{St}).
Since $\lambda \left( g\right) \Omega =\delta _{g}$, we have
$\overline{N\Omega }=\overline{\text{span}\left\{\delta _{g}:g\in
F\right\} }$. It therefore follows from $\left( D\lambda \left(
g\right) \right) \Omega =P\delta _{g}$ that
\begin{equation*}
D\lambda \left( g\right) =\left\{
\begin{array}{c}
\lambda \left( g\right) \text{ when }g\in F \\
0\text{ otherwise}
\end{array}
\right.
\end{equation*}
for all $g\in \Gamma $. From this it follows that we obtain a well
defined mapping
\begin{equation*}
E:=D|_{A}:A\rightarrow B
\end{equation*}
which is the required conditional expectation of $A$ onto $B$.

(ii) Now we show that $\left( A,\alpha\right) $ is $\left(
E,S\right) $-strongly mixing.

For any $f,g,h\in \Gamma$ we see that
\begin{equation*}
\left\langle \delta _{f},\alpha ^{n}\left( \lambda \left( g\right)
-E\lambda \left( g\right) \right) \delta _{h}\right\rangle =0
\end{equation*}%
for $n$ large enough, since it is zero for $g\in F$, while otherwise
$g$ has an infinite orbit which implies that there is no repetition
in the sequence $\left( T^{n}g\right) _{n\in \mathbb{N}}$, hence
\begin{equation*}
\left\langle \delta _{f},\alpha ^{n}\left( \lambda \left( g\right)
-E\lambda \left( g\right) \right) \delta _{h}\right\rangle
=\left\langle \delta _{f},\lambda \left( T^{n}g\right) \delta
_{h}\right\rangle =\left\langle \delta _{fh^{-1}},\delta
_{T^{n}g}\right\rangle
\end{equation*}
is non-zero for at most one value of $n$ (namely where
$T^{n}g=fh^{-1}$).
>From this it is straightforward to show that%
\begin{equation*}
\lim_{n\rightarrow \infty }\left\langle x,\alpha ^{n}\left(
a-Ea\right) x\right\rangle =0
\end{equation*}
for all $a\in A$ and all $x\in H$.

However, since $\Omega$ is cyclic and separating for $M$, every
normal state $\varphi$ on $M$ is given by $\varphi\left( a\right)
=\left\langle x,ax\right\rangle $ for some $x\in H$ (see for example
\cite[Theorem 2.5.31]{BR}), but the normal states on $M$ are exactly
the states given by density matrices on $H$ (see for example
\cite[Theorem 2.4.21]{BR}). This means that all the states on $M$
(and therefore on $A$) given by density matrices are already given
by the states of the form $\varphi\left( a\right)=\left\langle
x,ax\right\rangle $ with $x\in H$, proving the theorem.
\end{proof}

Note that one cannot replace the finite orbit algebra by the fixed
point algebra in this theorem, unless they happen to be the same,
i.e. when $T$ is such that all the orbits are either singletons or
infinite. However, the fixed point algebra is nevertheless still
important for other ergodic properties, even when the fixed point
algebra and finite orbit algebra are not equal, as the next result
illustrates. This result follows from a minor variation on arguments
from \cite{AD} using Haagerup's inequality (see \cite[Lemma 1.4]{Ha}
for the origins of this inequality):

\begin{proposition}
Let $\Gamma $ be the free group on a countably infinite set of
symbols $L$. Let $T$ be the automorphism of $\Gamma $ induced by any
bijection $L\rightarrow L$. Then the dual system $\left( A,\alpha
\right) $ of $\left(\Gamma ,T\right) $ is uniquely ergodic relative
to its fixed point algebra.
\end{proposition}

\begin{proof}
This follows by applying the same argument as in \cite[Section 2 and
Proposition 3.5]{AD}, using Haagerup's inequality, to the elements
of $\Gamma $ with infinite orbits under $T$, but also noticing that
\begin{equation*}
\lim_{N\rightarrow \infty }\frac{1}{N}\sum_{n=1}^{N}\alpha
^{n}\left( \lambda \left( g\right) \right)
\end{equation*}
exists for all $g\in F$, since the orbits of such $g$ are periodic,
from which we conclude that
\begin{equation*}
\lim_{N\rightarrow \infty }\frac{1}{N}\sum_{n=1}^{N}\alpha
^{n}\left(a\right)
\end{equation*}
exists for all $a\in A$. Then one simply applies \cite[Theorem
3.2]{AD} (or Theorem 2.15 above).
\end{proof}

Note that in the case where there are symbols in $L$ with finite
orbits of length more than $1$, the fixed point algebra differs from
the finite orbit algebra. By the proposition $\left( A,\alpha
\right) $ is nevertheless uniquely ergodic relative to its fixed
point algebra (even though not all the orbits in $\left( \Gamma
,T\right) $ are infinite or singletons), while at the same time it
is $\left( E,S\right) $-mixing as in Theorem 3.2, where $E$ does not
project onto the fixed point algebra, but rather onto the finite
orbit subalgebra.

\section{Unique $\left( E,S\right) $-ergodicity and unique
$\left( E,S\right) $-weak mixing}

 Let $(A,\a)$ be a $C^*$-dynamical
system, let $E:A\to A$ be a bounded linear operator, and let $S$
be a set of bounded linear functionals on $A$. In what follows, by $S(A)$ we denote the
set of all states defined on $A$.

\begin{definition}
A $C^*$-dynamical system $\left(  A,\alpha\right)  $ for the action
of a topological semigroup $G$ is called $S$\emph{-weakly amenable},
for a set $S\subset A^{\ast}$, if the following holds: There is a
right invariant measure $\rho$ on $G$, an averaging net $\left(
f_{\iota}\right)  $ for $\left(  G,\mathbb{C}\right)  $, and
$G\rightarrow\mathbb{C}:g\mapsto \varphi\left(  \alpha_{g}\left(
a\right)  \right)  $ is Borel measurable for every $a\in A$ and
$\varphi\in S$.
\end{definition}

\begin{definition} Let $(A,\a)$ be an $S$-weakly amenable $C^*$-dynamical system, with $G$ unimodular with
respect to the right measure $\rho$ and let $(f_\iota)$ be an
averaging net for $(G,\bc)$. Then $(A,\a)$ is said to be
\begin{enumerate}
\item[(i)]
\textit{ unique $(E,S)$-ergodic} w.r.t. $(f_\iota)$ if one has
\begin{equation} \label{mpp1}
\lim_{\iota}\frac{1}{\int f_\iota}\int f_\iota(g)
\f(\a_g(x-E(x)))dg=0\,, \ \ x\in A\,,\f\in S\,;
\end{equation}

\item[(ii)]
\textit{ unique $(E,S)$-weakly mixing} w.r.t. $(f_\iota)$ if one
has
\begin{equation} \label{mpp1}
\lim_{\iota}\frac{1}{\int f_\iota}\int f_\iota(g)
\big|\f(\a_g(x-E(x)))\big|dg= 0\,, \ \ x\in A\,,\f\in S\,.
\end{equation}
\end{enumerate}
 \end{definition}

\begin{remark}  Note that if one takes $S=S(A)$ and $E$ is invariant w.r.t. $\a$, i.e.
$\a_gE=E$ for all $g\in G$, then unique $(E,S)$-ergodicity (resp.
unique $(E,S)$-weak mixing ) coincides with unique ergodicity
relative to $A^\a$ (resp. unique $E$-weak mixing
\cite{FM},\cite{M3}) (see Theorem \ref{UE-A}). Note that in
\cite{FM1} other more complicated unique $E$-ergodic and unique
mixing $C^*$-dynamical systems arising from free probability have
been studied.
\end{remark}

Let us provide an example of unique $(E,S)$-weak mixing dynamical
system.

{\sc Example.} We consider $A=\mathbb{C}^4$, and $G=\bn$. Let
$$
\a= \left(
\begin{array}{cccc}
p & 1-p & 0 & 0\\
0 & 0 & 1  & 0 \\
0 & 1 & 0  & 0\\
0& 0& 0& 1 \\
\end{array}
\right).
$$

One can see that $\a$ has $\{p,1,-1,1\}$ eigenvalues with
corresponding eigenvectors
$\{\mathbf{e},\id,\mathbf{k},\mathbf{s}\}$, where
$$
{\mathbf{e}}=(1,0,0,0), \ \id=(1,1,1,1)\ \
{\mathbf{k}}=((p-1)/(p+1),1,-1,0),\ \ {\mathbf{s}}=(0,0,0,1).
$$
Due to independence the above vectors, any vector $\mathbf{x}$ in
$\bc$ can be represented by
$$
{\mathbf{x}}=\l\mathbf{e}+\m\id+\n\mathbf{k}+\t\mathbf{s}.
$$
Hence, one finds
$$
\a^n({\mathbf{x}})=\l
p^n{\mathbf{e}}+\m\id+\t\mathbf{s}+(-1)^n\n\mathbf{k}.
$$

>From this we immediately find that $\a$ is uniquely $E_\a$-ergodic,
where $E_\a$ is a projection from $\bc^4$ onto fixed point space
$\{\m\id+\t\mathbf{s}: \m,\t\in\bc\}$. Note that $\a$ is not
uniquely $E_\a$-weak mixing.

We let
$$
S=\{(0,x,x,y): \ x,y\geq 0, \ x+y=1\}.
$$

Define a mapping $E:\bc^4\to L$, where
$L=\{\m\id+\n\mathbf{k}\t\mathbf{s}: \m,\n,\t\in\bc\}$, by
$$
E({\mathbf{x}})=\m\id+\t\mathbf{s}+\n\mathbf{k}.
$$
Note that $E$ is a projection, not onto $Fix(\a)$.

 For every state
$\f\in S$ one can compute that
$$
\f(\a^n({\mathbf{x}}))= \m\f(\id)+\t\f({\mathbf{s}}), \ \ \forall
n\in\bn.
$$
Hence $\a$ is uniquely $(E,S)$-weak mixing.\\

\begin{proposition}\label{ch} Let $(A,\a)$ be an $S$-weakly amenable $C^*$-dynamical system with $G$
unimodular with respect to the measure $\rho$, and let $(f_\imath)$
be an averaging net for $(G,\bc)$. Assume that the linear hull of $S$ is
norm-dense in $A^*$ and $\a_gE=E$, for all $g\in G$, then following
conditions are equivalent:
\begin{enumerate}
\item[(i)] $(A, \a)$ is unique $(E,S)$-ergodic;

\item[(ii)]
$(A,\a)$ is unique ergodic relative to $A^\a$.
\end{enumerate}
\end{proposition}

The proof goes by the same lines with the proof of Theorem 2.1
\cite{FM}.

 Given a set $R$ in $A^*$ by $R^{ch}$ we
denote its convex hull.

\begin{proposition}\label{ch1} Let $(A,\a)$ be an $S$-weakly amenable $C^*$-dynamical system with $G$
unimodular with respect to the measure $\rho$, and let $(f_\imath)$
averaging net for $(G,\bc)$. The following conditions are
equivalent:
\begin{enumerate}
\item[(i)] $(A,\a)$ is unique $(E,S)$-weakly mixing
w.r.t. $(f_\imath)$;

\item[(ii)] $(A,\a)$ is unique
$(E,\overline{S^{ch}}^{\|\cdot\|_1})$-weakly mixing w.r.t.
$(f_\imath)$.

\item[(iii)] $(A,\a)$ is unique
$(E,\mathcal{S})$-weakly mixing w.r.t. $(f_\imath)$, where by
$\mathcal{S}$ we denote a closed subspace of $A^*$ generated by
$\overline{S^{ch}}^{\|\cdot\|_1}$.
\end{enumerate}
\end{proposition}

The proof is straightforward.

Recall that a linear functional $f\in A^*$ is called
\textit{Hermitian}, if $f^*=f$, where $f^*(x)=\overline{f(x^*)}$,
$x\in A$. In what follows we will use the third condition in
Proposition \ref{ch1}, i.e. $\mathcal{S}$ is
assumed to be any closed subspace of $A^*$.
By ${{\mathcal S}}_h$ we denote the set of all Hermitian functionals
belonging to $\mathcal{S}$. In some instances we will further assume that
the space $\mathcal{S}$ is self-adjoint, i.e. for every $f\in
\mathcal{S}$ one has $f^*\in \mathcal{S}$, in which case any functional belonging to
$\mathcal{S}$ is a linear combination of elements of $\mathcal{S}_h$.

In what follows, when we refer to the properties in Definition 4.2, we implicitly mean
with respect to the averaging net which is involved.

\begin{theorem} Let $(A,\a)$ be an $\mathcal{S}$-weakly amenable $C^*$-dynamical system
and let $(f_\imath)$ be an averaging net for $(G,\bc)$. Let the
dynamical system $(A\otimes A, \a\otimes \a)$ be unique $(E\otimes
E,\mathcal{S}\otimes\mathcal{S})$-ergodic, and assume that $E$
preserves the involution (i.e. $E(x^*)=E(x)^*$) and that $\mathcal{S}$ is
self-adjoint. Then $(A,\a)$ is unique $(E,\mathcal{S})$-weakly mixing.
\end{theorem}

\begin{proof}  Take any $\p\in
{\mathcal{S}}_h$, and $x\in A$ with $x=x^*$. Then unique $(E\otimes
E,{\mathcal{S}}\otimes{\mathcal{S}})$-ergodicity of $(A\otimes A,
\a\otimes \a)$ implies that
$$
\lim_{\imath}\frac{1}{\int f_\imath}\int f_\imath(g)
\p\otimes\p(\a_g\otimes \a_g((x-E(x))\otimes (x-E(x))))dg=0.
$$

Now self-adjointness of $x$ yields
\begin{equation}\label{ue2}
\lim_{\imath}\frac{1}{\int f_\imath}\int
f_\imath(g)|\p(\a_g(x-E(x)))|^2dg=0.
\end{equation}

By the Schwarz inequality one finds
\begin{eqnarray*}
\frac{1}{\int f_\imath}\int f_\imath(g) |\p(\a_g(x-E(x)))|dg&\leq&
\frac{1}{\int f_\imath}\sqrt{\int f_\imath}
\sqrt{\int f_\imath(g)|\p(\a_g(x-E(x)))|^2}dg\nonumber\\
&=& \sqrt{\frac{1}{\int f_\imath}\int f_\imath(g)
|\p(\a_g(x-E(x)))|^2dg},
\end{eqnarray*}
which with \eqref{ue2} implies
\begin{eqnarray}\label{eq3}
\lim_{\imath}\frac{1}{\int f_\imath}\int
f_\imath(g)|\p(\a_g(x-E(x)))|dg=0.
\end{eqnarray}

It is known that any functional in $\mathcal{S}$ can be represented
as a linear combination of functionals from ${\mathcal{S}}_{h}$,
therefore, from \eqref{eq3} one gets
\begin{eqnarray}\label{eq4}
\lim_{\imath}\frac{1}{\int f_\imath}\int
f_\imath(g)|\f(\a_g(x-E(x)))|dg=0, \ \ \ \ \textrm{for every } \ \
\f\in\mathcal{S},
\end{eqnarray}
which implies the desired assertion.
\end{proof}

\begin{theorem}\label{mix-a1} Let $(A,\a)$ and $(B,\b)$ respectively be
an $\mathcal{S}$-weakly amenable and an $\mathcal{H}$-weakly amenable
$C^*$-dynamical system, and let $(f_\imath)$ be an
averaging net for $(G,\bc)$. If $(A,\a)$ and $(B,\b)$ are unique
$(E_\a,\mathcal{S})$-weak mixing and unique
$(E_\b,\mathcal{H})$-weak mixing, respectively, then the
$C^*$-dynamical system $(A\otimes B, \a\otimes\b)$ is
$\mathcal{S}\otimes \mathcal{H}$-weakly amenable and unique
$(E_{\a}\otimes E_\b,\mathcal{S}\otimes\mathcal{H})$-weak mixing.
\end{theorem}

\begin{proof} The measurability $g\mapsto \f(\a_g(a))$ and $g\mapsto \p(\b_g(b))$
for every $\f\in\mathcal{S}$, $\p\in \mathcal{H}$, $a\in A$, $b\in
B$, implies that the measurability of $g\mapsto
\f(\a_g(a))\p(\b_g(b))$. Correspondingly, we find that $g\mapsto
\w(\a_g(a)\otimes\b_g(b))$ is measurable for every $\w\in
\mathcal{S}\odot\mathcal{H}$. The density
$\mathcal{S}\odot\mathcal{H}$ in $\mathcal{S}\otimes \mathcal{H}$
yields the measurability $g\mapsto \w(\a_g(a)\otimes\b_g(b))$ for
every $\w\in \mathcal{S}\otimes \mathcal{H}$. Using the same
argument, one can show the measurability of $g\mapsto
\w(\a_g\otimes\b_g(\xb))$ for all $\xb\in A\otimes B$. Hence,
$(A\otimes B, \a\otimes\b)$ is an $\mathcal{S}\otimes
\mathcal{H}$-weakly amenable.

Let $x\in A$ and $y\in B$. Denote
\begin{equation}\label{seq}
x_0=x-E_\a(x), \ \ \ y_0=y-E_\b(y).
\end{equation}
It is clear that $\|x_0\|\leq C_1\|x\|$, $\|y_0\|\leq C_2\|y\|$ for
some $C_1,C_2\in \br_+$. Moreover, for any $\f\in \mathcal{S}$,
$\v\in \mathcal{H}$ one can see that
\begin{eqnarray}\label{xyk1}
\lim_{\imath}\frac{1}{\int f_\imath}\int
f_\imath(g)|\f(\a_g(x_0))|dg=0, \ \ \lim_{\imath}\frac{1}{\int
f_\imath}\int f_\imath(g)|\v(\b_g(y_0))|dg=0
\end{eqnarray}

Consequently, the Schwarz inequality yields
\begin{eqnarray}\label{xyk2}
&&\frac{1}{\int f_\imath}\int f_\imath(g)|\f(\a_g(x_0))\v(\b_g(y_0))|dg\nonumber \\
&\leq&
\sqrt{\frac{1}{\int f_\imath}\int f_\imath(g)|\f(\a_g(x_0))|^2dg}
\sqrt{\frac{1}{\int f_\imath}\int f_\imath(g)|\v(\b_g(y_0))|^2dg}\nonumber\\
&\leq&
C_2\|y\|\|\v\|\sqrt{\frac{1}{\int f_\imath}\int f_\imath(g)|\f(\a_g(x_0))|^2dg}\nonumber\\
&\leq&
2\|y\|\|\v\|\big(C_1\|x\|\|\f\|\big)^{1/2}\sqrt{\frac{1}{\int
f_\imath}\int f_\imath(g)|\f(\a_g(x_0))|dg}
\end{eqnarray}

Hence, \eqref{xyk1} with \eqref{xyk2} implies that
\begin{eqnarray}\label{U3}
\frac{1}{\int f_\imath}\int f_\imath(g)|\f\otimes
\v(\a_g(x_0)\otimes \b_g(y_0))|dg=0.
\end{eqnarray}

Recall that finite linear combinations of elements of type
$\f\otimes\p$ we denote by $\mathcal{S}\odot\mathcal{H}$, i.e.
$$
\mathcal{S}\odot\mathcal{H}=\bigg\{\sum_{i=1}^n\l_i\f_i\otimes\p_i\
:\ \{\l_i\}_{i=1}^n\subset\bc,\
\{\f_i\}_{i=1}^n\subset{\mathcal{S}},
\{\p_i\}_{i=1}^n\subset{\mathcal{H}}, n\in\bn\bigg\}
$$

Consequently, from \eqref{U3} we find
\begin{eqnarray*}
\lim_{\imath}\frac{1}{\int f_\imath}\int
f_\imath(g)|\w(\a_g(x_0)\otimes \b_g(y_0))|dg=0 \ \ \textrm{for
any}\ \ \w\in \mathcal{S}\odot\mathcal{H}.
\end{eqnarray*}
This with \eqref{seq} means that
\begin{eqnarray}\label{uum1}
\lim_{\imath}\frac{1}{\int f_\imath}\int f_\imath(g)&&
\big|\w(\a_g(x)\otimes \b_g(y))-
\w(\a_g(x)\otimes \b_g(E_\b(y)))\nonumber\\
&&- \w(\a_g(E_\a(x))\otimes\b_g(y))+\w(\a_g(E_\a(x))\otimes
\b_g(E_\b(y)))\big|dg=0
\end{eqnarray}

Now take any functional $\w$ from $\mathcal{S}\odot\mathcal{H}$,
i.e. it has the following form
$$
\w=\sum_{i=1}^n\l_i\f_i\otimes\p_i, \ \f_i\in{\mathcal{S}},\
\p_i\in{\mathcal{H}}, \ i=1,\dots,n.
$$

Then we get
\begin{eqnarray}\label{eq12}
&&\big|\w(\a_g(x)\otimes\b_g(E(y)))-\w(\a_g(E_\a(x))\otimes\b_g(E(y))\big|\\[2mm]
&\leq&\sum_{i=1}^n\l_i\big|\f_i(\a_g(x))-\f_i(\a_g(E_\a(x))\big||\p_i(\b_g(E(y)))|\\[2mm]
&\leq& \big(\max_{1\leq i\leq n}\|\p_i\|\big)\|E_\b\|\|y\|
\sum_{i=1}^n\l_i\big|\f_i(\a_g(x-E_\a(x)))\big|
\end{eqnarray}

According to unique $(E_\a,\mathcal{S})$-weak mixing condition, from
the last relations one finds
\begin{eqnarray}\label{uum2}
\lim_{\imath}\frac{1}{\int f_\imath}\int
f_\imath(g)\big|\w(\a_g(x)\otimes \b_g(E_\b(y)))-
\w(\a_g(E_\a(x))\otimes\b_g(E_\b(y)))\big|dg=0
\end{eqnarray}

Similarly, one gets
\begin{eqnarray}\label{uum3}
\lim_{\imath}\frac{1}{\int f_\imath}\int
f_\imath(g)\big|\w(\a_g(E_\a(x))\otimes \b_g(y))-
\w(\a_g(E_\a(x))\otimes\b_g(E_\b(y)))\big|dg=0
\end{eqnarray}

The inequality
\begin{eqnarray*}
&&|\w\big(\a_g\otimes \b_g(x\otimes
y-E_\a\otimes E_\b(x\otimes y)\big)| \nonumber \\
&\leq&\big|\w(\a_g(x)\otimes \b_g(y))-
\w(\a_g(x)\otimes \b_g(E_\b(y)))\nonumber\\
&&- \w(\a_g(E_\a(x))\otimes\b_g(y))+\w(\a_g(E_\a(x))\otimes
\b_g(E_\b(y)))\big|\nonumber \\
&&+\big|\w(\a_g(x)\otimes \b_g(E_\b(y)))-\w(\a_g(E_\a(x))\otimes \b_g(E_\b(y)))\big|\\[2mm]
&&+\big|\w(\a_g(E_\a(x))\otimes \b_g(y))-\w(\a_g(E_\a(x))\otimes
\b_g(E_\b(y)))\big|
\end{eqnarray*}
with \eqref{uum1},\eqref{uum2}, \eqref{uum3} implies that
\begin{eqnarray}\label{eq21}
\lim_{\imath}\frac{1}{\int f_\imath}\int
f_\imath(g)\big|\w(\a_g\otimes \b_g(x\otimes y-E_\a\otimes
E_\b(x\otimes y))\big|dg=0.
\end{eqnarray}

The norm-denseness of the elements $\sum_{i=1}^m x_i\otimes y_i$ in
$A\otimes B$ with \eqref{eq21} yields
\begin{eqnarray}\label{mixxx}
\lim_{\imath}\frac{1}{\int f_\imath}\int
f_\imath(g)\big|\w(\a_g\otimes \b_g({\mathbf{z}}-E_\a\otimes
E_\b({\mathbf{z}}))\big|dg=0.
\end{eqnarray}
for arbitrary ${\mathbf{z}}\in A\otimes B$.

Now taking into account the norm denseness of
$\mathcal{S}\odot\mathcal{H}$ in $\mathcal{S}\otimes\mathcal{H}$,
we conclude from \eqref{mixxx} that $(A\otimes B,G,\a\otimes \b)$ is
unique $(E_\a\otimes E_\b,\mathcal{S}\otimes\mathcal{H})$-weak
mixing.
\end{proof}

\begin{remark} The proved Theorem \ref{mix-a1} extends some
results of \cite{M}-\cite{M3}. We note that in \cite{Av,L,Wa}
similar results were proved for weak mixing dynamical systems
defined over von Neumann algebras.
\end{remark}

\begin{theorem}\label{mix-erg} Let $(A,\a)$ and $(B,\b)$ respectively be
an $\mathcal{S}$-weakly amenable and an $\mathcal{H}$-weakly amenable
$C^*$-dynamical system, and let $(f_\imath)$ be an
averaging net for $(G,\bc)$. If $(A,\a)$ is unique
$(E_\a,\mathcal{S})$-weakly mixing and $(B,\b)$ unique
$(E_\b,\mathcal{H})$-ergodic with $\a_gE_\a=E_\a$ for all $g\in G$,
then the $C^*$-dynamical system $(A\otimes B,\a\otimes\b)$ is unique
$(E_{\a}\otimes E_\b,\mathcal{S}\otimes\mathcal{H})$-ergodic.
\end{theorem}

\begin{proof} Let $x\in A$ and
$y\in B$. Then using the same argument of the proof of Theorem
\ref{mix-a1} and equality $\a_gE_\a=E_\a$ one finds
\begin{eqnarray}\label{er1}
\lim_{\imath}\frac{1}{\int f_\imath}\int f_\imath(g)&&
\bigg(\w(\a_g(x)\otimes \b_g(y))-
\w(\a_g(x)\otimes \b_g(E_\b(y)))\nonumber\\
&&- \w(E_\a(x)\otimes\b_g(y))+\w(E_\a(x)\otimes
\b_g(E_\b(y)))\bigg)dg=0
\end{eqnarray}

\begin{eqnarray}\label{er2}
\lim_{\imath}\frac{1}{\int f_\imath}\int
f_\imath(g)\big|\w((\a_g(x-E_\a(x))\otimes \b_g(E_\b(y)))\big|dg=0, \\
\label{er3} \lim_{\imath}\frac{1}{\int f_\imath}\int
f_\imath(g)\big(\w(E_\a(x)\otimes (\b_g(y-E_\b(y)))\big)dg=0.
\end{eqnarray}
for any $\w\in{\mathcal{S}}\odot\mathcal{H}$.

 The following
inequality
\begin{eqnarray*}
&&\bigg|\frac{1}{\int f_\imath}\int f_\imath(g)\big(\w(\a_g\otimes
\b_g(x\otimes
y-E_\a(x)\otimes E_\b(y))\big)\bigg|dg \nonumber \\
&\leq&\bigg|\frac{1}{\int f_\imath}\int f_\imath(g)
\big(\w(\a_g(x)\otimes \b_g(y))-\w(\a_g(x)\otimes \b_g(E_\b(y)))dg\nonumber\\
&&-
\w(E_\a(x)\otimes \b_g(y))+\w(E_\a(x)\otimes \b_g(E_\b(y)))\big)\bigg|dg\\[2mm]
&&+\frac{1}{\int f_\imath}\int
f_\imath(g)\big|\w((\a_g(x-E_\a(x))\otimes
\b_g(E_\b(y)))\big|dg\\[2mm]
&&+ \frac{1}{\int f_\imath}\bigg|\int
f_\imath(g)\big(\w(E_\a(x)\otimes (\b_g(y-E_\b(y)))\big)dg\bigg|
\end{eqnarray*}
and  \eqref{er1}-\eqref{er3} implies that
\begin{eqnarray*}
\lim_{\imath}\frac{1}{\int f_\imath}\int f_\imath(g)
\big(\w(\a_g\otimes \b_g(x\otimes y-E_\a\otimes E_\b(x\otimes
y))\big)dg=0.
\end{eqnarray*}

Finally, the density argument shows that $(A\otimes B,\a\otimes\b))$
is unique $(E_a\otimes E_\b,\mathcal{S}\otimes\mathcal{H})$-ergodic.
\end{proof}

This theorem allows us to construct many nontrivial examples of
unique $(E,S)$-ergodic and unique $(E,S)$-mixing dynamical systems.
In particular, this theorem shows how unique $\left(E,S\right)$-weak
mixing of a system can be characterized in terms of unique
$\left(E,S\right) $-ergodicity of the product of the system with
itself, in analogy to the well-known result in classical ergodic
theory (see \cite{ALW}).

\section{Multitime correlations functions}

Higher order mixing and recurrence properties of C*- and W*-dynamical
systems have received attention lately, for example in \cite{NSZ}, \cite{D0}, \cite{Fid},
\cite{BDS} and \cite{AET} (much of which was inspired
by the work of Furstenberg in classical ergodic theory \cite{F1}). In the
case of a C*-dynamical system with an action of the group $\mathbb{Z}$ for
example, this involves multitime correlation functions of the form
\begin{equation*}
\varphi \left( \alpha ^{n_{1}}\left( a_{1}\right) ...\alpha ^{n_{k}}\left(
a_{k}\right) \right)
\end{equation*}
where $\varphi $ is a state on $A$. However in the papers mentioned, when
considering general systems, one is either limited to very low orders (i.e.
very small values of $k$), or to systems which are asymptotically abelian in
some sense. These limitations appear to be in the nature of things (see for
example the discussion in \cite[Section 1c]{AET}), but one can nevertheless
ask whether one has some positive results for special types of systems. This
is what we explore in this section, for the dual systems of Section 3.

We note that multitime correlation functions and related ergodic averages
have also been studied (both from a physical and a mathematical point of
view) in \cite{BF}, \cite{ABDF}, \cite{ABDF2}, \cite{Pe}, \cite{D1},
\cite{Fid0} and \cite{EK}. These papers provide further motivation for the work
in this section.

\begin{definition}
Let $\left( A,\alpha \right) $ be a C*-dynamical system for an action of the
semigroup $\mathbb{N}$. Let $E:A\rightarrow A$ be linear, and let $S$ be a
set of bounded linear functionals on $A$. Then $\left( A,\alpha \right) $ is said to be
$\left(E,S\right) $\emph{-mixing of order }$k$ if
\begin{equation*}
\lim_{\bar{n}^{(k)}\rightarrow \infty }\varphi \left( \alpha
^{n_{p(1)}}\left( a_{1}\right) ...\alpha ^{n_{p(k)}}\left( a_{k}\right)
-\alpha ^{n_{p(1)}}\left( Ea_{1}\right) ...\alpha ^{n_{p(k)}}\left(
Ea_{k}\right) \right) =0
\end{equation*}
for all $a_{1},...,a_{k}\in A$, all $p\in S_{k}$ and all $\varphi \in S$,
where $S_{k}$ is the group of all permutations of $\left\{ 1,...,k\right\} $
and $\bar{n}^{(k)}\rightarrow \infty $ means that
$n_{1},n_{2}-n_{1},n_{3}-n_{2},...,n_{k}-n_{k-1}\rightarrow \infty $. If this
holds for every $k=1,2,3,...$, we say that $\left( A,\alpha \right) $ is
$\left( E,S\right) $\emph{-mixing of all orders}.
\end{definition}

In the rest of this section we consider the following dual system: Let
$\Gamma $ be the free group on an arbitrary set of symbols $L$; we are
especially interested in the case where $L$ is not a finite set, but this
assumption is not necessary for the results below. We consider an
automorphism $T$ of $\Gamma $ induced by an arbitrary bijection
$L\rightarrow L$. Let $\left( A,\alpha \right) $ be the dual system of
$\left( \Gamma ,T\right) $. As in Section 3, we set $H=L^{2}\left( \Gamma
\right) $ and $F:=\left\{ g\in \Gamma :T^{\mathbb{N}}(g)\text{ is finite}\right\} $,
and the left regular representation of $\Gamma $ is denoted by $\lambda $.

Note that this system is not asymptotically abelian in the sense of
\cite[Definition 1.10]{AET}: In terms of the cyclic vector $\Omega =\delta
_{1}\in H$ we don't have
\begin{equation*}
\lim_{N\rightarrow \infty }\frac{1}{N}\sum_{n=1}^{N}\left\| \left[ \alpha
^{n}(a),b\right] \Omega \right\| =0
\end{equation*}
for all $a,b\in A$, in terms of the norm of $H$, as can be checked by
considering elements of the form $a=\lambda (g)$ and $b=\lambda (h)$.

At the same time the state preserving C*-dynamical system $\left( A,\alpha
,\mu \right) $ need not be ergodic or compact, where $\mu $ is the canonical
trace $\mu =\left\langle \Omega ,\left( \cdot \right) \Omega \right\rangle $.
This can be arranged by considering a $T$ which has some finite orbits on $L$,
making ergodicity impossible, as well as some infinite orbits (requiring
$L$ to be infinite), which makes compactness impossible. See for example
\cite[Theorems 3.4 and 3.6]{D2}, which covers the case of W*-dynamical
systems, but since ergodicity and compactness have Hilbert space
characterizations in terms of cyclic representations (i.e. the GNS
construction), those results hold for the C*-dynamical case as well.

The point is that $\left( A,\alpha \right) $, or more specifically
$\left(A,\alpha ,\mu \right)$, doesn't have the special properties assumed in
\cite{NSZ}, \cite{D0}, \cite{BDS} and \cite{AET} (and keep in mind that in
\cite{Fid} only low orders were considered).

\begin{theorem}
Let $E$ be the conditional expectation given by Theorem 3.2. Then
$\left(A,\alpha \right) $ is $\left( E,S\right) $-mixing of all orders, where $S$
is the set of all states on $A$ given by density matrices on $H$.
\end{theorem}

\begin{proof}
Consider any $f,g_{1},...,g_{k},h\in \Gamma $. First consider the case
$a_{j}=\lambda \left( g_{j}\right) $. When $g_{1},...,g_{k}\in F$, we know by
Theorem 3.2 that $Ea_{j}=a_{j}$ for all $j$, in which case the situation
becomes trivial.

So we assume that at least one $g_{j}$ is not in $F$, which means
$Ea_{j}=0$. In this case we are going to show that $\left\langle
\delta _{f},\alpha ^{n_{p(1)}}\left( a_{1}\right) ...\alpha
^{n_{p(k)}}\left( a_{k}\right) \delta _{h}\right\rangle =0$ for
$n_{1},n_{2}-n_{1},...,n_{k}-n_{k-1}$ large enough. To simplify the
notation we only do it for $p$ the identity, but all other
permutations work the same way. So we need to show that $\left(
T^{n_{1}}g_{1}\right) ...\left( T^{n_{k}}g_{k}\right) hf^{-1}\neq 1$
(the identity of $\Gamma $) for
$n_{1},n_{2}-n_{1},...,n_{k}-n_{k-1}$ large enough. To do this, we
will show that there is a symbol appearing in the (reduced) word
$\left( T^{n_{1}}g_{1}\right) ...\left( T^{n_{k}}g_{k}\right)
hf^{-1}$ which is not canceled by the other symbols in the word.
Consider the smallest $j$ such that $g_{j}$ has an infinite orbit
under $T$, so in particular $g_{j}$ contains a symbol $s$ with an
infinite orbit under $T$. (Note that $L$ only contains the
``original'' symbols, not their inverses, but here we allow for the
possibility that $s^{-1}\in L$, i.e. $s$ may be the inverse of a
symbol in $L$. More generally, the term symbol refers to elements of
$L\cup L^{-1}$.)

We note that for $n_{j}$ large enough, the symbol $T^{n_{j}}s^{-1}$
doesn't appear in $hf^{-1}$, since $hf^{-1}$ consists of only a
finite number of symbols while $s^{-1}$ has an infinite orbit under
$T$. So there is then no symbol in $hf^{-1}$ which can cancel the
$T^{n_{j}}s$ in $T^{n_{j}}g_{j}$. And since $g_{j}$ is a reduced
word, so is $T^{n_{j}}g_{j}$, since $T$ is a bijection, so since $s$
appears in $g_{j}$, the symbol $T^{n_{j}}s$ does appear in the
reduced word $T^{n_{j}}g_{j}$; i.e. the symbol $T^{n_{j}}s$ isn't
canceled in $T^{n_{j}}g_{j}$ itself .

If there are no other $g_{l}$ containing a symbol from the set $T^{\mathbb{Z}}s^{-1}$,
we are finished, since no symbol can then cancel $T^{n_{j}}s$, for
$n_{j}$ large enough. Suppose however that there is such a $g_{l}$, $l>j$.
The most general reduced forms of $g_{j}$ and $g_{l}$ are
\begin{equation*}
g_{j}=b_{1}s^{q_{1}}b_{2}s^{q_{2}}b_{3}...b_{t}s^{q_{t}}b_{t+1}
\end{equation*}
where $b_{1},...,b_{t+1}\in \Gamma $ don't contain $s$ (they could contain
$s^{-1}$ though), and
\begin{equation*}
g_{l}=c_{1}\left( T^{m_{1}}s^{-1}\right) ^{r_{1}}c_{2}\left(
T^{m_{2}}s^{-1}\right) ^{r_{2}}c_{3}...c_{u}\left( T^{m_{u}}s^{-1}\right)
^{r_{u}}c_{u+1}
\end{equation*}
where $c_{1},...,c_{u+1}\in \Gamma $ don't contain any $T^{n}s^{-1}$, with
$q_{1},...,q_{t},r_{1},...,r_{u}\in \mathbb{N=}\left\{ 1,2,3,...\right\} $
and $m_{1},...,m_{u}\in \mathbb{Z}$. From
\begin{equation*}
T^{n_{j}}g_{j}=\left( T^{n_{j}}b_{1}\right) \left( T^{n_{j}}s\right)
^{q_{1}}...\left( T^{n_{j}}s\right) ^{q_{t}}\left( T^{n_{j}}b_{t+1}\right)
\end{equation*}
and
\begin{equation*}
T^{n_{l}}g_{l}=\left( T^{n_{l}}c_{1}\right) \left(
T^{m_{1}+n_{l}}s^{-1}\right) ^{r_{1}}...\left( T^{m_{u}+n_{l}}s^{-1}\right)
^{r_{u}}\left( T^{n_{l}}c_{u+1}\right)
\end{equation*}
one then easily sees that for $n_{l}-n_{j}$ large enough, the inverse symbol
$T^{n_{j}}s^{-1}$ of $T^{n_{j}}s$, does not appear in $T^{n_{l}}g_{l}$,
since $s$ and $s^{-1}$ have infinite orbits under $T$ and so there are no
repetition in their orbits. This shows in particular that $\left(
T^{n_{1}}g_{1}\right) ...\left( T^{n_{k}}g_{k}\right) hf^{-1}\neq 1$ for
$n_{1},n_{2}-n_{1},...,n_{k}-n_{k-1}$ large enough, since the symbol
$T^{n_{j}}s$ in $T^{n_{j}}g_{j}$ isn't cancelled.

Note that nowhere did the order in which the $\alpha _{1},...,\alpha _{k}$
appear play any role in the arguments above, so we can restore the
permutation $p$. We have therefore shown that
\begin{equation*}
\left\langle \delta _{f},\left[ \alpha ^{n_{p(1)}}\left( a_{1}\right)
...\alpha ^{n_{p(k)}}\left( a_{k}\right) -\alpha ^{n_{p(1)}}\left(
Ea_{1}\right) ...\alpha ^{n_{p(k)}}\left( Ea_{k}\right) \right] \delta
_{h}\right\rangle =0
\end{equation*}
for $n_{1},n_{2}-n_{1},...,n_{k}-n_{k-1}$ large enough, for any
$f,g_{1},...,g_{k},h\in \Gamma $, where $a_{j}=\lambda \left( g_{j}\right) $.
The result now follows as in the proof of Theorem 3.2.
\end{proof}

Using this theorem, one can show the following results without too much
effort:

\begin{corollary}
For any $k\in \mathbb{N}$ we have
\begin{align*}
& \lim_{N\rightarrow \infty }\frac{1}{N}\sum_{n=1}^{N}\left| \varphi \left(
\alpha ^{p(1)n}(a_{1})\alpha ^{p(2)n}(a_{2})...\alpha ^{p(k)n}(a_{k})\right)
\right|  \\
& =\lim_{N\rightarrow \infty }\frac{1}{N}\sum_{n=1}^{N}\left| \varphi \left(
\alpha ^{p(1)n}(Ea_{1})\alpha ^{p(2)n}(Ea_{2})...\alpha
^{p(k)n}(Ea_{k})\right) \right|
\end{align*}%
for all $\varphi \in S$ (as in the theorem above), $p\in S_{k}$ and
$a_{1},...,a_{k}\in A$, and in particular these limits exist. The result
still holds without the absolute value signs.
\end{corollary}

\begin{proof}
The existence of
\begin{equation*}
\lim_{N\rightarrow \infty }\frac{1}{N}\sum_{n=1}^{N}\left| \varphi \left( \alpha
^{p(1)n}(Ea_{1})\alpha ^{p(2)n}(Ea_{2})...\alpha ^{p(k)n}(Ea_{k})\right)
\right|
\end{equation*}
can be shown easily from the fact that the orbits of elements of $F$ are
finite and therefore periodic under $T$ (similarly for the case where the
absolute value sign is dropped). The result then follows from the theorem
above.
\end{proof}

Note that the canonical trace $\mu =\left\langle \Omega ,\left( \cdot
\right) \Omega \right\rangle $ is $\alpha $-invariant, and then we obtain
the more conventional form of Furstenberg's Recurrence Theorem (see \cite{F1}
and \cite{HK05}):

\begin{corollary}
For any $k\in \mathbb{N}$ we have
\begin{equation*}
\lim_{N\rightarrow \infty }\frac{1}{N}\sum_{n=1}^{N}\left| \mu \left(
a\alpha ^{n}(a)\alpha ^{2n}(a)...\alpha ^{kn}(a)\right) \right| >0
\end{equation*}
and any positive $a\in A$ such that $\mu (a)>0$. This is still true if
$a,\alpha ^{n}(a),...,\alpha ^{kn}(a)$ appear in any other order.
\end{corollary}

\begin{proof}
The limit exists by the previous corollary and the fact that $\mu $ is $\alpha $-invariant.
Let $B_{0}$ be the unital $\ast $-algebra generated by $\left\{\lambda (g):g\in F\right\} $
and consider $b_{0},...,b_{k}\in B_{0}$. Since
the orbits in $F$ are finite, there is an $N_{0}>0$ such that $\alpha
^{jn}(b_{j})=b_{j}$ for $j=1,...,k$ and all $n\in N_{0}\mathbb{Z}$. So, for
$a_{0},...,a_{k}$ in the C*-subalgebra $B$ of $A$ generated by $B_{0}$, and
for arbitrary $\varepsilon >0$, there is an $N_{0}>0$ such that
\begin{equation*}
\left\| a_{0}\alpha ^{n}(a_{1})\alpha ^{2n}(a_{2})...\alpha
^{kn}(a_{k})-a_{0}...a_{k}\right\| <\varepsilon
\end{equation*}
and hence
\begin{equation*}
\left| \mu \left( a_{0}\alpha ^{n}(a_{1})\alpha ^{2n}(a_{2})...\alpha
^{kn}(a_{k})\right) \right| >\left| \mu \left( a_{0}...a_{k}\right) \right|
-\varepsilon
\end{equation*}
for all $n\in N_{0}\mathbb{Z}$. Since $Ea\geq 0$ and $\mu (Ea)=\mu (a)>0$,
we have $\mu \left( (Ea)^{k+1}\right) >0$, hence the result follows by
taking $a_{j}=Ea$ and $\varepsilon <\mu \left( (Ea)^{k+1}\right) $, and
applying the previous corollary (again keeping in mind $\mu \circ \alpha
=\mu $).
\end{proof}

The following corollary is a version of Bergelson's Theorem \cite{B} (an
extension of which was obtained in \cite{HK05}) in classical ergodic theory.
A noncommutative version for asymptotically abelian weakly mixing systems
was studied in \cite{D0}.

\begin{corollary}
For any $p,q\in \mathbb{Z}$
\begin{align*}
& \lim_{N\rightarrow \infty }\frac{1}{N^{2}}\sum_{m=p+1}^{p+N}
\sum_{n=q+1}^{q+N}\left| \mu \left( a_{0}\alpha ^{m}(a_{1})\alpha
^{n}(a_{2})\alpha ^{m+n}(a_{3})\right) \right|  \\
& =\lim_{N\rightarrow \infty }\frac{1}{N^{2}}\sum_{m=p+1}^{p+N}
\sum_{n=q+1}^{q+N}\left| \mu \left( \left( Ea_{0}\right) \alpha
^{m}(Ea_{1})\alpha ^{n}(Ea_{2})\alpha ^{m+n}(Ea_{3})\right) \right|
\end{align*}
for all and $a_{0},\dots,a_{3}\in A$.
\end{corollary}

\begin{proof}
This is similar to the proofs above, with some small complications and
corresponding modifications to account for the fact that here we sum over a
square.
\end{proof}

In the language of classical ergodic theory, we can say that the
C*-dynamical system $\left( EA,\alpha |_{EA}\right) $ is a characteristic
factor of $\left( A,\alpha \right) $ for the Furstenberg and Bergelson
averages above (see for example \cite[Section 1.3]{HK05}). However, here we
in effect considered a more topological version with many states, including
states that need not be $\alpha $-invariant, while \cite{HK05} considered
the measure theoretic version with a fixed invariant state (given by an
invariant measure).

\subsection*{Acknowledgments}

The research of the first named author (R.D.) is supported by the
National Research Foundation of South Africa. The second named
author (F.M.) acknowledges the MOHE Grants FRGS11-022-0170,
ERGS13-024-0057. He also thanks the Junior Associate scheme of the
Abdus Salam International Centre for Theoretical Physics, Trieste,
Italy.

\end{document}